\documentclass[12pt,dvips]{amsart}
\usepackage{amsfonts, amssymb, latexsym, epsfig}
\usepackage{euler, amsfonts, amssymb, latexsym, epsfig, epic}

\theoremstyle{plain}
  \newtheorem{theorem}{Theorem}[section]
  \newtheorem{proposition}[theorem]{Proposition}
  \newtheorem{lemma}[theorem]{Lemma}
  \newtheorem{corollary}[theorem]{Corollary}
  \newtheorem{conjecture}[theorem]{Conjecture}
\theoremstyle{definition}
  \newtheorem{definition}[theorem]{Definition}
  \newtheorem{example}[theorem]{Example}
  
  \newtheorem{question}[theorem]{Question}
 \theoremstyle{remark}

\numberwithin{equation}{section}

\def\ZZ{{\mathbb Z}}
\def\QQ{{\mathbb Q}}
\def\RR{{\mathbb R}}

\def\CC{{\mathbb C}}
\def\Proj{{\mathbb P}}

\def\Ess{{\mathcal{E}}}

\def\schubp{\mathfrak{S}}
\def\hh{\mathfrak{h}}
\def\bb{\mathfrak{b}}
\def\gg{\mathfrak{g}}

\def\Sym{\mathrm{Sym}}
\def\Des{\mathrm{Des}}

\def\xx{\mathbf{x}}

\newcommand{\qbin}[2]{\left[ \begin{matrix} #1 \\ #2 \end{matrix} \right]_q}

\begin{document}
\pagestyle{plain}
\title[Presenting Schubert variety cohomology]{Presenting the cohomology of a Schubert variety}

\author{Victor Reiner}
\address{School of Mathematics\\
University of  Minnesota\\
Minneapolis, MN 55455}
\email{reiner@math.umn.edu}

\author{Alexander Woo}
\address{Dept. of Mathematics, Statistics, and Computer Science\\
Saint Olaf College\\
Northfield, MN 55057 }
\email{woo@stolaf.edu}

\author{Alexander Yong}
\address{Dept. of Mathematics\\
University of Illinois at Urbana-Champaign\\
Urbana, IL 61801}
\email{ayong@illinois.edu}

\subjclass[2000]{14M15, 14N15}

\keywords{Schubert calculus, Schubert variety, cohomology presentation, bigrassmannian, essential set}

\date{\today}

\begin{abstract}
We extend the short presentation due to [Borel '53] of the cohomology
ring of a generalized flag manifold to a relatively short presentation
of the cohomology of any of its Schubert varieties. Our result is
stated in a root-system uniform manner by introducing the {\it
essential set of a Coxeter group element}, generalizing and giving a
new characterization of [Fulton '92]'s definition for permutations.
Further refinements are obtained in type $A$.
\end{abstract}

\maketitle

\tableofcontents

\newpage 
\section{Introduction}
\label{intro-section}


The cohomology of the generalized flag manifold $G/B$, for any complex
semisimple algebraic group $G$ and Borel subgroup $B$, has a classical
presentation due to Borel \cite{Borel}.  Pick a maximal torus
$T\subset B\subset G$, choose a field $\Bbbk$ of characteristic zero, and let
$V=\Bbbk\otimes_{\ZZ} X(T)$, where $X(T)$ is the coweight lattice of
$T$.  Let $\Bbbk[V]:=\Sym(V^{\star})$ be the symmetric algebra for
$V^{\star}$; it naturally carries an action of the Weyl group $W :=
N_G(T)/T$.  Furthermore let $\Bbbk[V]^W$ be the ring of $W$-invariants
of $\Bbbk[V]$, and $\Bbbk[V]^W_+$ the ideal of $\Bbbk[V]$ generated by
the $W$-invariants of positive degree.  A classical theorem of Borel
\cite{Borel} states that
\begin{equation}
\label{Borel-presentation}
H^*(G/B,\Bbbk) \cong \Bbbk[V]/(\Bbbk[V]^W_+).
\end{equation}

A desirable feature of Borel's presentation is its shortness: since $W$
acts on $V^\star$ as a finite reflection group, by Chevalley's Theorem,
the $W$-invariants $\Bbbk[V]^W$ are a polynomial algebra
$\Bbbk[f_1,\ldots,f_n]$, where $n:=\dim_\Bbbk V$ is the rank of
$G$ \cite[Section~3.5]{humphreys}.  Hence Borel's presentation shows
that $H^\star(G/B)$ is a complete intersection as it has $n$ generators and
$n$ relations.

 There is a second way to describe $H^{\star}(G/B)$ arising from
 a cell decomposition of $G/B$. One has
$$
\begin{array}{rll}
G&=\bigsqcup_{w \in W} BwB & \text{(Bruhat decomposition)}\\
G/B&=\bigsqcup_{w \in W} X_w^\circ & \text{(Schubert cell decomposition), with}\\ 
X_w^\circ&:=BwB/B \cong \CC^{\ell(w)} &\text{ (open Schubert cell)}.  
\end{array}
$$ Here $\ell(w)$ denotes the Coxeter group length of $w$ under the
Coxeter system $(W,S)$, where $S$ are the reflections in the simple
roots $\Pi$ among the positive roots $\Phi^+$ 
within the natural root system $\Phi$ associated to our Lie pinning.  
This gives a
$CW$-decomposition for $G/B$ in which all cells have even real
dimension.  Hence the fundamental homology classes $[X_w]$ of their
closures $X_w:=\overline{X_w^\circ}$, the {\bf Schubert varieties},
form a $\ZZ$-basis for the integral homology $H_\star(G/B)$, and their
(Kronecker) duals $\sigma_w:=[X_w]^\star$ form the dual $\ZZ$-basis for
the integral cohomology $H^\star(G/B)$.

The above facts lead to an {\it a priori} presentation
for $H^\star(X_w)$; see also \cite[Cor. 4.4]{Carrell} and
\cite[Prop. 2.1]{GasharovReiner}:  The Schubert variety $X_w$ is the
union of the Schubert cells $X_u^\circ$ for which $u \leq w$ in the
{\it (strong) Bruhat order}; hence it inherits a cell decomposition
from the flag variety.  Consequently, the map on cohomology
$$
H^\star(G/B) \rightarrow H^\star(X_w)
$$
induced by including $X_w \hookrightarrow G/B$ is surjective 
with kernel 
\begin{equation}
\label{first-I-def}
I_w:=\Bbbk\text{-span of }\{ \sigma_u: u \not\leq w \}.
\end{equation}
This gives rise to presentations
\begin{equation}
\label{wasteful-presentation}
\begin{aligned}
H^\star(X_w) &\cong  H^\star(G/B)/I_w &\text{ working over }\ZZ\\ 
         &\cong  \Bbbk[V]/\left( I_w + \Bbbk[V]^W_+ \right)&
\text{ working over }\Bbbk.\\ 
\end{aligned}
\end{equation}

This presentation \eqref{wasteful-presentation} involves a generating set \eqref{first-I-def}
for $I_w$ with at most $|W|$ generators.  However, this generating set
for $I_w$ is wasteful in that it not only generates $I_w$ as an ideal but
actually spans it $\Bbbk$-\emph{linearly} within $H^\star(G/B)$. 

Therefore, a basic question is to request more efficient generating
sets for $I_w$.  For type $A_{n-1}$, earlier work
\cite{GasharovReiner} reduced the $|W|=n!$ upper bound on the number
of generators for $I_w$ to a polynomial bound of $n^2$ for the class
of Schubert varieties $X_w$ \emph{defined by inclusions}; this class
includes all smooth $X_w$. Moreover, for a certain subclass of smooth
Schubert varieties $X_w$ considered originally by Ding
\cite{Ding1,Ding2}, they gave a smaller generating set for these $I_w$
having only $n$ generators.  This latter result was applied in
\cite{DevelinMartinReiner} to classify these varieties up to
isomorphism and homeomorphism (the generating set \eqref{first-I-def}
having proved too unwieldy). One motivation for this work arises from
the desire to extend this classification to general Schubert varieties of
type $A$.  Experience suggests that presentations for $I_w$ that are
as simple as possible are best for this purpose.

The question of finding 
simple presentations of $I_w$ for other root systems appears to 
have been less studied.

The main goal of this paper is to give a concise and root-system
uniform extension of Borel's presentation that produces for arbitrary
$w\in W$ an abbreviated list of generators for $I_w$.  Our first main
result, Theorem~\ref{general-bound}, achieves this via a strong
restriction on the {\bf descent set}
$$
\Des(u):=\{ s \in S: \ell(us) < \ell(u) \}
$$ of the elements $u$ in $W$ that index elements in our list of ideal
generators $\sigma_u$ for $I_w$.  We will need the following
definitions, the first two of which are standard:
\begin{enumerate}
\item[$\bullet$]
An element $u\in W$ is {\bf grassmannian} if $|\Des(u)| \leq 1$.
\item[$\bullet$]
An element $v \in W$ is {\bf bigrassmannian} if both $v$ and $v^{-1}$ are grassmannian.  
\item[$\bullet$]
Given $w \in W$, the {\bf essential set for} $w$, denoted
$\Ess(w)$, is the set of $u\in W$ which are minimal in the Bruhat order among those
{\it not} below $w$.
\end{enumerate}
The nomenclature ``essential set'' for $\Ess(w)$ is justified in
Proposition~\ref{essential-proposition}, where we give a new
characterization of Fulton's essential set \cite{Fulton} for the case
of the symmetric group $W=S_n$.  Indeed, $\Ess(w)$ has been
previously studied from a different point of view: a result of Lascoux
and Sch\"utzenberger \cite[Th\'eor\`eme 3.6]{LascouxSchutzenberger},
implicit in
our first main result below, is that the elements in
$\Ess(w)$ are bigrassmannian for any $w$ in $W$.\footnote{ For Lascoux
and Sch\"utzenberger this arises in their work (see also Geck and Kim
\cite{GeckKim} and Reading \cite{Reading}) seeking efficient {\it
encodings} of the strong Bruhat order; it is perhaps not surprising that
it should arise in our search for efficient cohomology presentations
as well.}

\begin{theorem}
\label{general-bound}
For any $w\in W$, working in a field $\Bbbk$ of characteristic zero,
the ideal $I_w$ defining $H^\star(X_w)$ as a quotient of $H^\star(G/B)$ is
generated by the cohomology classes $\sigma_u$ where $u \not\leq w$
and $u$ is grassmannian.

More precisely, $I_w$ is generated by the classes $\sigma_u$ indexed
by those grassmannian $u$ for which there exist some bigrassmannian
$v$ in $\Ess(w)$ satisfying both $u \geq v$ and $\Des(u) = \Des(v)$.
\end{theorem}

In type $A$, a similar result was obtained by Akyildiz, Lascoux, and
Pragacz \cite[Theorem~2.2]{ALP}. Specifically, they prove the first
sentence of Theorem~\ref{general-bound}, though they do not address the
strengthening given by the second sentence. Their methods are mainly
geometric, as opposed to our essentially combinatorial arguments. Their
work provides, to our knowledge, the first inroads towards an abbreviated
generating set for $I_w$.

Theorem~\ref{general-bound} replaces the general upper bound
of $|W|$ on the number of generators needed for $I_w$ with the bound 
\begin{equation}
\label{less-wasteful-general-bound}
\sum_{s \in S} [W:W_{S\setminus\{s\}}],
\end{equation}
where, for any subset $J\subset S$, $W_J$ denotes the parabolic
subgroup of $W$ generated by $J$.  Our theorem is deduced in
Section~\ref{Hiller-section} from a more general result
(Theorem~\ref{Hiller-BGG-general-bound}) that applies to Hiller's
extension \cite[Chapter IV]{Hiller} of Schubert calculus as introduced
by Bernstein, Gelfand, and Gelfand \cite{BGG} and
Demazure~\cite{Demazure} to the \emph{coinvariant algebras} of finite
reflection groups $W$.

Section~\ref{Hiller-section} explains and proves Theorem~\ref{general-bound}.
In Section~\ref{parabolic-section}, we exploit the particular form of our
generators to derive a straightforward extension to Schubert varieties
in any partial flag manifold $G/P$ associated to 
a parabolic subgroup $P$ of $G$.

Section~\ref{type-A-section} examines
more closely Theorem~\ref{general-bound} in 
type $A_{n-1}$.  Here one can take $G=GL_n(\CC)$,
with $B$ the subgroup of invertible upper triangular matrices, 
$T$ the invertible 
diagonal matrices, and $W=S_n$ the permutation matrices. The bound 
on generators
for $I_w$ in \eqref{less-wasteful-general-bound} becomes $2^n$, 
at least a practical improvement on $|W|=n!$.
%
%
More importantly,
one can be even more explicit and efficient in the
generating sets for $I_w$. 

Identify points of 
$G/B$ with complete flags 
$$
\langle 0 \rangle \subset V_1 \subset \cdots \subset V_{n-1} \subset \CC^n.
$$ 
Under this identification, each Schubert variety $X_w$ is the set of
flags satisfying certain specific conditions derived from $w$ of the
form $\dim_\CC (V_r \cap \CC^s)\geq t$.  The bigrassmannians $v$
comprising $\Ess(w)$ correspond to Fulton's {\it essential Schubert
  conditions}, a minimal list of such conditions defining the Schubert
variety $X_w$.  Our second main result (Theorem~\ref{J-generators})
provides for each bigrassmannian $v$ in $W=S_n$ a generating set for the ideal
\begin{equation}
\label{eqn:J_vgenset}
J_v:=\Bbbk\text{-span of }\{ \sigma_u: u \geq v \}
\end{equation}
in type $A$ that is smaller than the one used as a general
step (Theorem~\ref{general-J-generators}) in the proof of Theorem~\ref{general-bound}
for arbitrary finite Coxeter groups $W$.  Our proof of 
Theorem~\ref{J-generators} is based on symmetric function identities that we devise for this
purpose.

Therefore, concatenating these generating sets for $J_v$ gives a generating set for
\begin{equation}
\label{eqn:I_wconcat}
I_w = \sum_{v \in \Ess(w)} J_v
\end{equation}
that is smaller than the one provided by Theorem~\ref{general-bound}.
We remark that this result subsumes (and slightly improves upon;
see Example~\ref{GR-comparison}) 
the generating set of size $n^2$ given by \cite{GasharovReiner} 
in the case of Schubert varieties defined by inclusions.

Actually, we conjecture that this smaller generating set for $J_v$ in
type $A$ is minimal (although the generating set (\ref{eqn:I_wconcat})
for $I_w$ obtained by concatenation is not always minimal; see
Example~\ref{non-minimal-example}.)  The significance of this
minimality conjecture, as explained in Section~\ref{minconj}, is that
it implies an exponential \emph{lower} bound of at least
\[\binom{n/2}{n/4} \sim \frac{{\sqrt 2}^{n+2}}{\sqrt{\pi n}}\]
on the number of generators needed for $I_w$, accompanying our
exponential upper bound of $2^n$.  Thus one would not be able to expect
short presentations for $H^*(X_w)$ in general, at least in type~$A$.

\section{Proof of Theorem~\ref{general-bound}}
\label{Hiller-section}

As in Section~1, let $\Bbbk$ be a field of characteristic zero.
For $v,w$ in $W$, define two $\Bbbk$-linear subspaces $J_v, I_w$ of the cohomology
$H^\star(G/B)$ with $\Bbbk$ coefficients:
\begin{equation}
\label{I-J-definitions}
\begin{aligned}
J_v&:=\Bbbk\text{-span of }\{ \sigma_u: u \geq v \}\\
I_w&:=\Bbbk\text{-span of }\{ \sigma_u: u \not\leq w\} 
\left(=\sum_{v \in \Ess(w)} J_v\right).
\end{aligned}
\end{equation}
Recall the essential set $\Ess(w)$ is the 
set of all Bruhat-minimal elements of $\{v \in W:  v \not\leq w\}$.

The Schubert cell decomposition of $G/B$ shows that both of these 
$\Bbbk$-subspaces
are actually ideals in the cohomology ring $H^\star(G/B)$:  one has that
$I_w$ (respectively $J_v$) is
the kernel of the surjection $H^\star(G/B) \rightarrow H^\star(X)$ 
induced by the inclusion of the
$B$-stable subvariety $X \subseteq G/B$, where $X=X_w$ (respectively
$X= \bigcup_{u \not\geq v} X_u$).

Also recall from Section~1 the following important property of
bigrassmannians in the Bruhat order on Coxeter groups, originally due
to Lascoux and Sch\"utzenberger \cite[Th\'eor\`eme
3.6]{LascouxSchutzenberger} and Geck and Kim \cite[Lemma 2.3 and
Theorem 2.5]{GeckKim}.

\begin{lemma}
\label{bigrassmannian-fact}
For any Coxeter system $(W,S)$ and $w$ in $W$, every element of
$\Ess(w)$ is bigrassmannian.
\end{lemma}

\noindent
See Section~\ref{essential-subsection} below for a further
interpretation of $\Ess(w)$ when $W$ is a Weyl group of type
$A_{n-1}$.  As a consequence of Lemma~\ref{bigrassmannian-fact} and
(\ref{I-J-definitions}), finding generators of $J_v$ for
bigrassmannian $v$ automatically gives generators for the ideals $I_w$.

We will actually work at the level of generality of irreducible
Coxeter systems $(W,S)$ with $W$ finite, using Hiller's version
\cite{Hiller} of the Schubert calculus \cite{BGG, Demazure} for
coinvariant algebras. This emphasizes that the arguments of this
section and the next only depend on Coxeter combinatorics and formal
properties of divided difference operators and coinvariant
algebras.  We review here the relevant facts from \cite[Chapter
IV]{Hiller}.

Let $W$ be a finite and irreducible Coxeter group, and $V$ its
reflection representation.  One then picks a (possibly
non-crystallographic) root system $\Phi$ for $W$ as follows: $\Phi
\subset V^\star$ is any $W$-stable choice\footnote{Note that this may
require coefficients in a subfield $\Bbbk$ of $\RR$ strictly larger
than $\QQ$ when $W$ is not crystallographic.} of a set of a linear
functionals $\alpha$ and $-\alpha$ such that the perpendicular spaces
$\alpha^\perp$ in $V$ run through the reflecting hyperplanes of $W$.
These reflecting hyperplanes divide $V$ into chambers; we pick one
which we call the dominant chamber $\mathcal{C}$.  One declares the
positive roots $\Phi^+\subseteq\Phi$ to be the roots which are
positive on this chamber.  The Coxeter generators $S$ of $W$ are
declared to be the reflections across the hyperplanes which are facets of
$\mathcal{C}$.  Among the positive roots are the simple roots
$\Pi=\{\alpha_s\}_{s \in S}$, where $\alpha_s$ is the positive root
vanishing on the reflecting hyperplane of $s$.

  With these choices, one defines for each $w\in W$ the {\bf Hiller
Schubert polynomial}\footnote{``BGG/Demazure Schubert polynomial''
would be as fair, but for us the main point is to distinguish these
polynomials from the type $A$ Lascoux-Sch\"{u}tzenberger version
forthcoming in Section~4, which are not the same as elements of
$k[V]$, although equivalent mod $k[V]^W_+$.}
$S_w$ in $\Bbbk[V]=\Sym(V^\star)$ by
$$
S_w := \partial_{w^{-1}w_0} S_{w_0},
$$
where $w_0$ is the unique longest element in $W$,
for which one declares that
$$
S_{w_0} :=  \frac{1}{|W|}\prod_{\alpha \in \Phi^+} \alpha.
$$ 
The images of the polynomials $S_w$ as $w$ ranges over all of $W$ will
form a basis for the {\bf coinvariant algebra}
$\Bbbk[V]/(\Bbbk[V]^W_+)$.  Here the {\bf divided difference}
operators $\partial_u$ on $\Bbbk[V]$ are defined first for $s\in S$ by
$$
\partial_s (f):= \frac{f-s(f)}{\alpha_s}
$$
and then for any $u\in W$ of Coxeter length $\ell=\ell(u)$ by
$$
\partial_u:=\partial_{s_{i_1}} \cdots \partial_{s_{i_\ell}},
$$
where $u=s_{i_1} \cdots s_{i_\ell}$ is any choice of a (reduced)
decomposition expressing $u$ in terms of the generators $s\in S$.

The relation to a generalized flag manifold $G/B$ is that $G$, $B$,
and $T$ come equipped with a (crystallographic) root system and Weyl
group $W$. Let $\hh\subset \bb \subset \gg$ be the Lie algebras associated
to $T\subset B\subset G$.
Taking the reflection representation $V$ to be $V=\hh$ (or
more generally $V=\Bbbk \otimes_\ZZ X(T)$) and the root system $\Phi$
to be the set of weights of the adjoint representation of $\gg$ acting
on itself, with $\Phi^+$ the weights of the action of $\gg$ on
$\bb$.  In this case, it was proven in \cite{BGG,
Demazure} that the element $S_w\in\Bbbk[\hh]$ is a lift under the
surjection
$$
\Bbbk[\hh] \quad \rightarrow \quad \Bbbk[\hh]/(\Bbbk[\hh]^W_+) 
\quad \left( \cong H^\star(G/B) \right)
$$ of the Schubert cohomology class $\sigma_w$ in $H^{\ell(w)}(G/B)$
which is {\it Kronecker dual} to the fundamental homology class
$[X_w]$ of $X_w$ in $H_{\ell(w)}(G/B)$ and {\it Poincar\'e dual} to
the fundamental homology class $[X_{w_0w}]$ in
$H_{\ell(w_0)-\ell(w)}(G/B)$.

  We collect basic properties of the divided difference operators $\partial_s$ and 
Schubert polynomials $S_w$ that we need below.
\begin{enumerate}
\item[(a)] {\bf Leibniz rule}:  $\partial_s (fg) = \partial_s(f) \cdot g + s(f) \partial_s(g).$
\item[(b)] $\partial_s(f)=0$ if and only if $s(f)=f$.  In the case where $f=S_w$ for some $w$, then
$\partial_s(S_w)=0$ if and only if $\ell(ws) > \ell(w)$, or equivalently if and only if $s \not\in \Des(w)$.
\item[(c)] Consequently, $\partial_s$ preserves the ideal 
$I= \Bbbk[V]^W_+$ generated
by the $W$-invariants of positive degree in $\Bbbk[V]$.
\item[(d)] 
$$
\partial_u \partial_v = 
  \begin{cases} 
    \partial_{uv}& \text{ if }\ell(u)+\ell(v)=\ell(uv)\\
    0          & \text{ otherwise.}\\
  \end{cases}
$$
\item[(e)] If $\ell(w)=\ell(w')$ then $\partial_{w'} S_w =
\delta_{w,w'}$.  (By definition, $\delta_{a,b}=1$ if $a=b$ and
$\delta_{a,b}=0$ otherwise.)
\item[(f)] Consequently, the {\bf Schubert structure constants} $c_{u,v}^w$ 
which uniquely express
$$
S_u S_v = \sum_{\substack{w \in W:\\ \ell(w)=\ell(u)+\ell(v)}} 
   c_{u,v}^w  S_w \qquad \mod k[V]^W_+
$$
can be computed by the formula $c_{u,v}^w= \partial_w \left( S_u S_v \right)$.
\item[(g)] These structure constants also satisfy this sparsity rule:
$$
c_{u,v}^w=0 \text{  unless }w \geq u,v\text{ in Bruhat order.}
$$  
This is a consequence of the Pieri
formula \cite[\S IV.3]{Hiller} for multiplying the $S_u$ by any of the
degree one elements that generate $\Bbbk[V]/(\Bbbk[V]^W_+)$.
\end{enumerate}

The key to the proof of Theorem~\ref{general-bound} turns
out to be the following lemma (perhaps of independent interest)
about some further sparsity of the Schubert structure
constants $c_{u,v}^w$ appearing in property $(f)$ above.
Given $J \subset S$, recall 
that every $w$ in $W$ has a unique length-additive 
{\bf parabolic factorization} 
$$
w=u \cdot x
$$ 
where $x$ lies in the parabolic subgroup $W_J$ generated by $J$,
and $u$ lies in the set $W^J$ of minimum-length coset representatives
for $W/W_J$, characterized by the property that \linebreak
$\Des(u) \cap J = \varnothing$.

\begin{lemma}
\label{sparsity-lemma}
%
Let $J \subset S$.  Suppose $w,w'$ in $W$ lie in the same coset $wW_J=w'W_J$,
so that $w = u \cdot x$ and $w'=u \cdot x'$ for some $u$ in $W^J$ and $x, x'$ in $W_J$.  
Then
$
c_{u,x}^{w'}= \delta_{w',w} = \delta_{x',x}.
$
\end{lemma}
\begin{proof}
Using property (f) of divided differences above, one can rephrase
the lemma as saying that, for any $w'\in W$ with $w'=u \cdot x'$ where
$x' \in W_J$ and $\ell(x')=\ell(x)$ (so that $\ell(w')=\ell(w)$), one
has
\begin{equation}
\label{rephrased-sparsity}
\partial_{w'} (S_u S_x) = \delta_{x,x'}.
\end{equation}

We prove \eqref{rephrased-sparsity} by induction on $\ell(x')$.  In the base case
where $\ell(x')=0$, one has $x'=x=1$ and $w'=w=u$, so the assertion
\eqref{rephrased-sparsity} follows from property (e) above.

In the inductive step, let $\ell(x')>0$. Thus, there exists $s\in J$ with
$\ell(x's)<\ell(x')$.  Consequently, 
$$
\ell(w')=\ell(u)+\ell(x') = \ell(u)+\ell(x's)+\ell(s).
$$ 
Hence by properties (d), (a) and (b), one has, respectively, 
\begin{equation}
\label{Leibniz-induction}
\begin{aligned}
\partial_{w'}( S_u S_x ) 
  & =  \partial_u \partial_{x's} \partial_s ( S_u S_x ) \\
  & =  \partial_u \partial_{x's} \left( \partial_s(S_u) \cdot S_x 
                                + s(S_u) \partial_s(S_x ) \right)\\
  & =  \partial_u \partial_{x's} ( S_u \partial_s(S_x ) ).
\end{aligned}
\end{equation}

Now consider two cases.

\vskip .1in
\noindent
{\sf Case 1.} $\ell(xs) > \ell(x)$.  Then
property (b) says that $\partial_s(S_x)=0$.  Using this in the last
line of \eqref{Leibniz-induction}, one concludes
that
$$ 
\partial_{w'}( S_u S_x ) = 0= \delta_{x,x'}
$$
since $\ell(x's)< \ell(x')$ implies $x \neq x'$.  

\vskip .1in
\noindent
{\sf Case 2.} $\ell(xs) < \ell(x)$.  Then $\partial_s(S_x)=S_{xs}$,
and 
$$
\begin{aligned}
\partial_{w'}( S_u S_x ) 
    & =  \partial_u \partial_{x's} ( S_u \partial_s(S_x ) ) \\
    & =  \partial_u \partial_{x's} ( S_u S_{xs} ) )\\
    & = \delta_{xs,x's} \\
    & = \delta_{x,x'},
\end{aligned}
$$
where the second-to-last equality applied the inductive hypothesis
to $x's$ and $xs$.
\end{proof}

We now use this lemma to find a smaller generating set for the ideal
$J_v$ as defined in \eqref{I-J-definitions} based on the descent set
$\Des(v)$.  Working more generally, for any finite Coxeter group $W$
and a choice of root data for the Hiller Schubert calculus,
define the following two $\Bbbk$-subspaces within the coinvariant
algebra $\Bbbk[V]/(\Bbbk[V]^W_+)$:
\begin{equation}
\label{Hiller-I-and-J-definitions}
\begin{aligned}
J_v&:=\Bbbk\text{-span of }\{ S_u: u \geq v \}\\
I_w&:=\Bbbk\text{-span of }\{ S_u: u \not\leq w\} = \sum_{v\in\Ess(w)} J_v.
\end{aligned}
\end{equation}
Note that in this context we can appeal to property (g) to see that
$J_v$ and $I_w$ are actually {\it ideals} within the coinvariant algebra
$\Bbbk[V]/(\Bbbk[V]^W_+)$.

\begin{theorem}
\label{general-J-generators}
Let $v$ be an element of a finite Coxeter group $W$ and $J$ a subset
of $S$.  Assume $v$ lies in $W^J$, or, equivalently, $\Des(v) \cap J =
\varnothing$.

Then $J_v$ is generated as an ideal within the coinvariant algebra 
$\Bbbk[V]/(\Bbbk[V]^W_+)$
by the set 
$$
\{ S_u: u \in W^J, \,\,  u \geq v \}.
$$
\end{theorem}
\begin{proof}
For any $t\in W^J$, let $J^\prime_t$ denote the ideal 
$$
\{ S_u: u\in W^J, \,\ u\geq t \}.
$$
We need to show $J^\prime_v=J_v$.
Certainly $J^\prime_v \subseteq J_v$ by definition, so it remains to
show the reverse inclusion.

Proceed by induction on the {\bf colength} $\ell(w_0)-\ell(v)$ of $v$.
In the base case where $v$ has colength $0$, $v=w_0$; therefore the
assumption $v \in W^J$ implies $J=\varnothing$, so $W^J=W$, and there
is nothing to prove.

In the inductive step, given $w \geq v$, one must show that $S_w$ lies
in $J'_v$.  Factor $w=u\cdot x$ uniquely with $u \in W^J$ and $x \in W_J$.
We will use repeatedly 
the fact (see \cite[\S 2.5]{BjornerBrenti}) that the map 
$$
\begin{aligned}
W & \overset{P^J}{\longrightarrow} W^J \\
w &\longmapsto u
\end{aligned}
$$
is order-preserving for the Bruhat order.  In particular, since it
was assumed that $w \geq v$ above, one has $u \geq v$ here.

By Lemma~\ref{sparsity-lemma}, one has
$$
\begin{aligned}
S_u S_x &= S_w + \sum_{w'} c_{u,x}^{w'} S_{w'},\,\,\text{ so that } \\
S_w &= S_u S_x - \sum_{w'} c_{u,x}^{w'} S_{w'}.
\end{aligned}
$$ Here each $w'$ appearing in the sums satisfies $w' \geq u$ by
property (g), and hence if one factors $w'=u' \cdot x'$ with $u' \in
W^J$ and $x' \in W_J$, then $w' \geq u$ implies $u' \geq u$.  But then
Lemma~\ref{sparsity-lemma} also says that $c_{u,x}^{w'}=0$ unless $u'
> u\ (\geq v)$; hence, by induction, for any $w^\prime$ with
$S_{w^\prime}$ appearing with nonzero coefficient in the right hand
sum, $J_{u^\prime}=J^\prime_{u^\prime}$ holds for the corresponding
$u^\prime$ in the factorization $w^\prime=u^\prime\cdot x^\prime$.  By
definition, $J'_{u'} \subset J'_v$, and since $S_u$ also lies in
$J'_v$ as $u \in W^J$, one concludes that $S_w$ lies in $J'_v$ as
desired.
\end{proof}

Theorem~\ref{general-bound} is then a special case of the following result, which 
is immediate from Lemma~\ref{bigrassmannian-fact} and Theorem~\ref{general-J-generators}:

\begin{theorem}
\label{Hiller-BGG-general-bound}
Let $W$ be a Coxeter system $(W,S)$ with $W$ finite, and let $w$ be an element
of $W$.  Then the ideal $I_w$ 
of the coinvariant algebra $\Bbbk[V]/(\Bbbk[V]^W_+)$, defined as in 
\eqref{Hiller-I-and-J-definitions}, is
generated by the Schubert polynomials 
$$
\{S_u: u \not\leq w, \text{ and }u\text{ is grassmannian}\}.
$$

More precisely, $I_w$ is generated by the set
$\{S_u \}$ for those $u$ for which there exist some (bigrassmannian)
$v$ in $\Ess(w)$ satisfying both $u \geq v$ and $\Des(u) = \Des(v)$.
\end{theorem}

\section{Reducing presentations in $H^{\star}(G/P)$ to those in $H^{\star}(G/B)$}
\label{parabolic-section}

  We explain in this section how Theorem~\ref{Hiller-BGG-general-bound}
leads to a shorter presentation more generally for the cohomology of Schubert
varieties in any partial flag manifold.

 Given $J \subset S$, one has the parabolic subgroup $W_J$ of $W$
generated by $J$.  One also has a corresponding parabolic subgroup $P_J$ of $G$,
generated by the Borel subgroup $B$ together with representatives within $G$
that lift the elements $J \subset W=N_G(T)/T$.
The Borel picture identifies the cohomology ring
of the {\bf (generalized) partial flag manifold}  $G/P_J$ as the subring of
$W_J$-invariants inside the cohomology of $G/B$.  In other words, 
the quotient map 
\[G/B \overset{\pi}{\rightarrow} G/P_J\] 
induces the inclusion
\begin{equation}
\label{parabolic-inclusion}
H^\star(G/P_J) \cong H^\star(G/B)^{W_J} \overset{i}{\hookrightarrow} H^\star(G/B).
\end{equation}

Recall that $W^J$ denotes the set of minimum-length coset
representatives for $W/W_J$.  Inside $H^\star(G/B)$, the cohomology
classes $\{ \sigma_w: w \in W^J\}$ lie in this $W_J$-invariant
subalgebra $H^\star(G/B)^{W_J}$ and form a $\Bbbk$-basis identified
with the $\Bbbk$-basis of Schubert cohomology classes $\sigma_{wW_J}$
for $H^\star(G/P_J)$.  One also has that the pre-image of Schubert
varieties in $G/P_J$ are certain Schubert varieties of $G/B$:
specifically,
$$
\pi^{-1}(X_{wW_J}) = X_{w_{\max}},
$$
where $w_{\max}$ is the unique {\it maximum}-length coset representative in $wW_J$.

  Working more generally with the Hiller Schubert calculus for any finite Coxeter
system $(W,S)$, the inclusion \eqref{parabolic-inclusion} generalizes to the inclusion
\begin{equation}
\label{Hiller-parabolic-inclusion}
\Bbbk[V]^{W_J}/ \Bbbk[V]^W_+ \Bbbk[V]^{W_J} \cong 
  \left( \Bbbk[V] / ( \Bbbk[V]^W_+ ) \right)^{W_J}
    \overset{i}{\hookrightarrow} \Bbbk[V]/(\Bbbk[V]^W_+).
\end{equation}
The first isomorphism shown in \eqref{Hiller-parabolic-inclusion} 
is a consequence of the fact that one has an {\it averaging} map 
\begin{equation}
\label{parabolic-retraction}
\begin{array}{rll}
\Bbbk[V] &\overset{\rho}{\longrightarrow} &\Bbbk[V]^{W_J} \\
f &\longmapsto &\frac{1}{|W_J|} \sum_{w \in W_J} w(f)
\end{array}
\end{equation}
which provides a $\Bbbk[V]^{W_J}$-linear (and hence also
$\Bbbk[V]^W$-linear) retraction map for $i$, meaning that $\rho \circ
i = 1_{\Bbbk[V]^{W_J}}$.  Inside $\Bbbk[V]/(\Bbbk[V]^W_+)$, the Hiller
Schubert polynomials $\{ S_w : w \in W^J\}$ lie in this
$W_J$-invariant subalgebra $\left( \Bbbk[V] / ( \Bbbk[V]^W_+ )
\right)^{W_J}$ and provide a $\Bbbk$-basis for it.

  The retraction in \eqref{parabolic-retraction} 
also provides the relation between the cohomology presentations
for the Schubert varieties $X_{wW_J}$ and $X_{w_{\max}}$.  Recall that when
one has an inclusion of rings $R \overset{i}{\hookrightarrow} \hat{R}$, one can
relate ideals of $R$ and $\hat{R}$ by the operations of {\it extension} and {\it contraction}:
given an ideal $I$ in $R$, its extension $\hat{R}I$ to $\hat{R}$ is the ideal it generates
in $\hat{R}$, and given an ideal $\hat{I}$ of $\hat{R}$, its contraction to $R$ is the intersection $\hat{I} \cap R$.
Say that the inclusion 
$R \overset{i}{\hookrightarrow} \hat{R}$ is a {\it split inclusion} if it has
an $R$-linear {\it retraction}  $\hat{R} \overset{\rho}{\rightarrow} R$, 
meaning that $\rho \circ i = 1_R$.  The following
proposition about this situation is straightforward and well-known.

\begin{proposition}
\label{retraction-proposition}
Assume $R \overset{i}{\hookrightarrow} \hat{R}$ is a split inclusion,
and $\hat{I}$ is an ideal of $\hat{R}$ which is generated by its contraction $I:=\hat{I} \cap R$ to $R$.

Then a set of elements $\{g_\alpha\}$ lying in $R$ generate $\hat{I}$ as an ideal of $\hat{R}$
if and only if the same elements $\{g_\alpha\}$ generate $I=\hat{I}\cap R$ as an ideal of $R$.
\end{proposition}

We will apply this proposition to the split inclusion \eqref{Hiller-parabolic-inclusion}
and these ideals 
\begin{equation}
\label{parabolic-ideal-definition}
\begin{array}{llll}
I=I_{wW_J} &:= \Bbbk\text{-span of }\{S_u: u \in W^J, \,\, uW_J \not\leq wW_J\} 
          &\subset  \left( \Bbbk[V] / ( \Bbbk[V]^W_+ ) \right)^{W_J}& = R\\
{\hat I} = I_{w_{\max}} &:= \Bbbk\text{-span of }\{S_u: u \in W,  \,\, u \not\leq w_{\max}\} 
                &\subset  \Bbbk[V] / ( \Bbbk[V]^W_+ ) & ={\hat R}
\end{array}
\end{equation}
which in the case where $(W,S)$ comes from an algebraic group $G$ have the following interpretations
as kernels:
\begin{equation}
\label{kernel-interpretations}
\begin{aligned}
I_{wW_J} &=\ker\left( H^\star(G/P_J) \overset{i^\star}{\rightarrow} H^\star(X_{wW_J}) \right) \\
I_{w_{\max}}&=\ker\left( H^\star(G/B) \overset{i^\star}{\rightarrow} H^\star(X_{w_{\max}}) \right).
\end{aligned}
\end{equation}
Borel's picture already gives a very short presentation for
$H^\star(G/P_J)$ or more generally $\left( \Bbbk[V] / ( \Bbbk[V]^W_+ )
\right)^{W_J}$, as we now explain.  The isomorphism in
\eqref{Hiller-parabolic-inclusion} says that a presentation of $\left(
\Bbbk[V] / ( \Bbbk[V]^W_+ ) \right)^{W_J}$ is equivalent to a
presentation of the quotient $\Bbbk[V]^{W_J}/ \Bbbk[V]^W_+
\Bbbk[V]^{W_J}$.  Since both $W_J$ and $W$ are finite reflection groups
acting on $V$, their invariant rings are both polynomial algebras
$$
\begin{aligned}
\Bbbk[V]^{W_J}&=\Bbbk[g_1,\ldots,g_n] \\
\Bbbk[V]^W &=\Bbbk[f_1,\ldots,f_n],\\
\end{aligned}
$$
and hence the quotient can be presented as a graded complete intersection ring:
$$
\begin{aligned}
\left( \Bbbk[V] / ( \Bbbk[V]^W_+ ) \right)^{W_J}
  &\cong \Bbbk[V]^{W_J}/ \Bbbk[V]^W_+ \Bbbk[V]^{W_J} \\
  &\cong \Bbbk[g_1,\ldots,g_n]/ (f_1,\ldots,f_n).
\end{aligned}
$$
Thus we only need to provide generators for the ideal $I_{wW_J}$.

\begin{theorem}
\label{parabolic-generators-theorem}
Let $(W,S)$ be a finite Coxeter system, with $J \subseteq S$ and $w$ in $W$,
and $w_{\max}$ the unique maximum-length representative of $wW_J$. 
Consider the inclusion
$$
R:=\left( \Bbbk[V] / ( \Bbbk[V]^W_+ ) \right)^{W_J} 
\quad
\overset{i}{\hookrightarrow}
\quad
\hat{R}:=\Bbbk[V] / ( \Bbbk[V]^W_+ ).
$$
Then the following hold.
\begin{enumerate}
\item[(i)] The essential set $\Ess(w_{\max})$ lies entirely in $W^J$.
\item[(ii)] The ideal $I_{w_{\max}}$ of $\hat{R}$ is generated by
its contraction $I_{w_{\max}} \cap R$.
\item[(iii)] This contraction is the same as $I_{wW_J}$. 
\item[(iv)] The set 
$$
\bigcup_{v \in \Ess(w)} \{S_u: u \geq v, \Des(u) = \Des(v)\}
$$
both generates $I_{w_{\max}}$ as an ideal of $\hat{R}$ and also generates
the contraction $I_{wW_J}$ as an ideal of~$R$.
\end{enumerate}
\end{theorem}
\begin{proof}
For assertion (i), assume for the sake of contradiction that $v\in
\Ess(w_{\max})$ and that $vs < v$ for some $s$ in $J$.  Since $v$ is
Bruhat-minimal among the elements {\it not} below $w_{\max}$, this
implies $vs \leq w_{\max}$.  However $w_{\max} s < w_{\max}$ by
maximality of $w_{\max}$ within $wW_J$, so the {\it lifting property}
\cite[Prop. 2.2.7, Cor. 2.2.8]{BjornerBrenti} of Bruhat order implies
that $v \leq w_{\max}$, a contradiction.

For assertion (ii), apply Theorem~\ref{Hiller-BGG-general-bound} to
$w_{\max}$ to conclude that $I_w$ is generated by the set
$\{S_u\}$ for those $u$ for which there exist $v\in
\Ess(w_{\max})$ satisfying both $u \geq v$ and $\Des(u) \subseteq
\Des(v)$.  By (i), this forces $u$ to lie in $W^J$, so that $S_u$ is
$W^J$-invariant and therefore lies in $R$.

For assertion (iii), from the definition
\eqref{parabolic-ideal-definition} and the fact that $R$ has a
$\Bbbk$-basis given by $\{S_u: u \in W^J\}$, it suffices to show that
for any $u$ in $W^J$, one has $uW_J \leq wW_J$ if and only if $u \leq
w_{\max}$.  By definition, $uW_J \leq wW_J$ if and only if $u \leq
w_{\min}$, where $w_{\min}$ is the unique representative of $wW_J$
lying in $W^J$.  The usual parabolic factorization $W=W^J \cdot W_J$
allows one to write down a reduced word $\omega$ for $w_{\max}$ in the
concatenated form
$$
\omega=\omega_1 \cdot \omega_2
$$
where the prefix $\omega_1$ factors $w_{\min}$, and the suffix
$\omega_2$ contains only generators in $J$.  The {\it subword
characterization} of Bruhat order \cite[Cor. 2.2.3]{BjornerBrenti}
shows that $u \leq w_{\max}$ if and only if $u$ is factored by a
reduced subword of this word $\omega$; this subword must necessarily use no
generators from $J$ (since $u$ is in $W^J$) and hence must actually be a
subword of $\omega_1$.  Thus $u \leq w_{\max}$ if and only if $u \leq
w_{\min}$, as desired.

Assertion (iv) then follows from assertions (ii), (iii) and Proposition~\ref{retraction-proposition}.
\end{proof}

\section{Refinements in Type $A$}
\label{type-A-section}

We investigate further the situation when $W$ is a Weyl group of 
type $A_{n-1}$, which exhibits extra features; one can: 
\begin{enumerate}
\item[$\bullet$] be more explicit about bigrassmannians and their essential sets $\Ess(w)$,
\item[$\bullet$] produce even smaller generating sets for the ideals
$I_w$ and $J_v$, which are conjecturally minimal in the case of $J_v$, and
\item[$\bullet$] work with $\ZZ$ coefficients rather than
over a field $\Bbbk$ of characteristic zero.
\end{enumerate}

\subsection{Schubert conditions and bigrassmannians}
\label{bigrassmannian-subsection}

  In type $A_{n-1}$, points in the variety $G/B$ 
are identified with complete flags of subspaces
$$
\langle 0\rangle \subset V_1 \subset \cdots \subset V_{n-1} \subset \CC^n
$$ having $\dim_\CC V_i = i$.  Pick as our particular base flag $\{
\CC^i \}_{i=0}^n$, where $\CC^i$ is spanned by the first $i$ standard
basis elements; this flag is fixed by the Borel subgroup $B$
consisting of the invertible upper-triangular matrices within
$G:=GL_n(\CC)$.  Picking the maximal torus $T$ of invertible diagonal
matrices, one identifies the Weyl group $W=N_G(T)/T$ with the
symmetric group $S_n$.  The Coxeter generators $S$ for $W=S_n$
associated to our Borel subgroup $B$ is the set of adjacent
transpositions $S=\{(1\leftrightarrow 2), (2\leftrightarrow
3),\ldots,(n-1 \leftrightarrow n)\}$.

  The Schubert variety $X_w$ corresponding to a
permutation $w=w_1 w_2 \cdots w_n\in S_n$ (written in one-line notation)
can be defined as the subvariety
of flags satisfying the conjunction of the {\bf Schubert conditions}
\begin{equation}
\label{Schubert-condition}
\dim_\CC (V_r \cap \CC^s) \geq t
\end{equation}
where 
\begin{equation}
\label{eqn:rankfn}
t={t}_{r,s}(w):=|\{w_1,w_2,\ldots,w_r\} \cap \{1,2,\ldots,s\}|
\end{equation}
for all $r,s=1,2,\ldots,n-1$ is the {\bf rank function} associated to $w$.  
Denote the condition \eqref{Schubert-condition}
by $C_{r,s,t}$ (for arbitrary $t$, not necessarily of the form (\ref{eqn:rankfn})).  
Note that $C_{r,s,t}$ is vacuous unless $t > r+s-n$.

The following 
explicit identification of bigrassmannian permutations 
is well-known and straightforward. 

\begin{lemma}
The bigrassmannian permutations (other than the identity) are
parameterized by $r,s,t$ with $1\leq t \leq r,s \leq n$ and $t>r+s-n$.
Let $v_{r,s,t,n}$ denote the unique bigrassmannian permutation
$v_1\ldots v_n\in S_n$ such that
\begin{enumerate}
\item[$\bullet$] $\Des(v)=\{(r\leftrightarrow r+1)\}$ 
\item[$\bullet$] $\Des(v^{-1})=\{(s\leftrightarrow s+1)\}$, and
\item[$\bullet$] $v_t=s+1$.
\end{enumerate}
Then explicitly we have:
\begin{equation}
\label{v-explicitly}
\begin{aligned}
v_{r,s,t,n}&:=(1, 2, \ldots,t-1 , \\
        & \qquad s+1, s+2, \ldots, s+r-t+1, \\
        & \quad \qquad t, t+1,t+2, \ldots,s,\\
        & \qquad \qquad s+r-t+2,s+r-t+3,\ldots,n).
\end{aligned}
\end{equation}
\end{lemma}


There is a simple relation between these Schubert conditions $C_{r,s,t}$
and the bigrassmannian permutations in $W=S_n$:

\begin{proposition}
\label{bigrassmannian-as-Schubert-condition}
Let $w\in S_n$.
Then the Schubert condition $C_{r,s,t}$ is satisfied by all flags in $X_w$
if and only if $v_{r,s,t,n} \not\leq w$.
\end{proposition}
\begin{proof}
Note $C_{r,s,t}$ is a Schubert condition on $X_w$ if and only
if $t_{r,s}(w) \geq t$.  It is then straightforward to check that the latter is
equivalent to $v_{r,s,t,n} \not\leq w$ using the {\it tableau criterion}
\cite[Theorem 2.6.3]{BjornerBrenti} for comparing elements in the Bruhat ordering.
\end{proof}

Note that imposing an arbitrary conjunction of Schubert conditions on 
complete flags cuts out a $B$-stable subvariety of $G/B$, 
but this subvariety may be a {\it reducible}
union of Schubert varieties rather than a single Schubert variety $X_w$.
However, in type $A$, when one imposes a {\it single} Schubert condition, the result is
always a (single) Schubert variety. This fact can be
traced to special properties of the Bruhat order in type $A$,
%
%
first identified by Lascoux and Sch\"utzenberger \cite{LascouxSchutzenberger}, 
and exploited further by Geck and Kim \cite{GeckKim} and Reading \cite{Reading}. 
To explain this, we first recall some terminology.

\begin{definition}
In a poset $P$, say that an element $v$ is a {\it dissector} of $P$ if
there exists a (necessarily unique) element $w$ in $P$ for which $P$
decomposes as the disjoint union of the principal order filter above $v$ and the
principal order ideal below $w$:
$$
P= \{u \in P: u \geq v \} \quad \bigsqcup \quad \{u \in P: u \leq w \}.                                                              
$$
Say that an element $a$ in a poset $P$ (which need not be a lattice) is
{\it join-irreducible}
if there does not exist a subset $X \subset P$ with $a \not\in X$
such that $a$ is the least element among all upper bounds for $X$ in $P$.
\end{definition}

There are two subtle issues to point out in this definition of
join-irreducibles.  Firstly, when the finite poset is a {\it lattice}, an element
is join-irreducible if and only if it covers a unique element.  
However, for non-lattices, one can have join-irreducibles 
that cover more than one element.  
For example, the strong Bruhat order in type $A_2$ has four non-minimal, non-maximal elements,
each of which is join-irreducible, but two of them cover two elements.
Secondly, all of the posets that we will consider have a unique least element,
(e.g. in Bruhat order on $W$, the least element is the identity of $W$), and this least element
is {\it not} considered join-irreducible because it is the least element among all upper
bounds for the empty set $X=\varnothing$.

\begin{theorem}\cite{LascouxSchutzenberger, GeckKim, Reading}
\label{dissective-theorem}
\begin{enumerate}
\item[(i)]
In any finite poset, every dissector is join-irreducible.  When the
poset is the Bruhat order for a Coxeter system $(W,S)$ of type $A, B, H_3,H_4$ or
$I_2(m)$, the converse holds:  the join-irreducible elements are exactly the dissectors.
\item[(ii)]
In the Bruhat order for any finite Coxeter system $(W,S)$, every join-irreducible element is bigrassmannian.
In type $A$, the converse holds:  the (non-identity) bigrassmannian elements are exactly the join-irreducibles.
%
\end{enumerate}
In particular, in type $A$, for every bigrassmannian $v$ in $W$, there exists a (necessarily unique) element
$w$ in $W$ for which $\Ess(w)=\{v\}$.
\end{theorem}

\begin{corollary}
In type $A_{n-1}$, a single Schubert condition $C_{r,s,t}$ on the flags
in $G/B$ cuts out the Schubert variety $X_{w_{r,s,t,n}}$ in $G/B$, where
$w_{r,s,t,n}$ is the unique element with $\Ess(w_{r,s,t,n})=\{v_{r,s,t,n}\}$ 
as in Theorem~\ref{dissective-theorem}.

Thus in type $A_{n-1}$, for any bigrassmannian 
$v_{r,s,t,n}$, one has equality of
the two ideals $J_{v_{r,s,t,n}}=I_{w_{r,s,t,n}}$ within the coinvariant algebra
$\Bbbk[V]/(\Bbbk[V]^W_+)$.
\end{corollary}

\noindent
We remark that, as with $v_{r,s,t,n}$, one knows $w_{r,s,t,n}$ explicitly 
(see \cite[\S 8]{Reading}):
$$
\begin{aligned}
w_{r,s,t,n} &=(n, n-1, \ldots,(n-r+t+1), \\
          &  \qquad s, s-1, \ldots, s-t+1, \\
           & \quad \qquad n-r+t,n-r+t-3, \ldots,s+1,\\
           & \qquad \qquad s-t,s-t-1,\ldots,1).
\end{aligned}
$$

\subsection{Bigrassmannians and essential Schubert conditions}
\label{essential-subsection}

Next, we explain the relation between what we have called the essential set $\Ess(w)$
for $w$ and Fulton's essential set of Schubert conditions for $X_w$.

Note that there are implications among the various Schubert conditions $C_{r,s,t}$.
Fulton introduced the {\bf essential set} of a permutation, a set of coordinates 
$\{(r_i,s_i)\}\subset n\times n$ which give an inclusion-minimal subset of 
Schubert conditions $C_{r_i,s_i,t}$ with $t=t_{r_i,s_i}(w)$ that suffice to define
$X_w$ as a subset of the flag manifold. (See further remarks in Example~\ref{exa:essential_convention}
below.) Correspondingly, we call these Schubert conditions the
{\bf essential Schubert conditions} for
$X_w$; see \cite[\S3]{Fulton}, \cite[pp. 20-21]{FultonPragacz}, and
\cite[\S2]{ErikssonLinusson}.

\begin{proposition}
\label{essential-proposition}
The Schubert condition $C_{r,s,t}$ implies the Schubert condition $C_{r',s',t'}$
if and only if  $v_{r,s,t,n} \leq v_{r',s',t',n}$ in Bruhat order.

Therefore in type $A_{n-1}$, Fulton's essential set of Schubert conditions $C_{r,s,t}$ for $X_w$ 
correspond bijectively to the elements of the essential set $\Ess(w)$ for $w$ defined for a general
Coxeter group.
\end{proposition}
\begin{proof}
For the first assertion, note that the Schubert cell decomposition for $X_w$ and
Theorem~\ref{dissective-theorem} give the following:
$$
\begin{aligned}
 C_{r,s,t} \text{ implies }C_{r',s',t'} 
 \Leftrightarrow& \qquad X_{w_{r,s,t,n}} \subseteq X_{w_{r',s',t',n}} \\
 \Leftrightarrow& \qquad \{u \in W: u \leq w_{r,s,t,n}  \} 
                               \subseteq \{u \in W: u \leq w_{r',s',t',n}\}\\
 \Leftrightarrow& \qquad \{u \in W: u \geq v_{r,s,t,n}  \} 
                               \supseteq \{u \in W: u \geq v_{r',s',t',n}\}\\
 \Leftrightarrow& \qquad v_{r,s,t,n} \leq v_{r',s',t',n}.
\end{aligned}
$$
The second assertion follows immediately from the first.
\end{proof}

\begin{example}
\label{exa:essential_convention}
In order to be explicit about the bijection asserted in
Proposition~\ref{essential-proposition}, it will be convenient for us
to use a slight adaptation of Fulton's essential set. This bijection
can be inferred from the discussion in \cite{FultonPragacz} and
\cite{GasharovReiner} (our conventions are in line with those the
latter text), and one also can thereby also give an explicit bijection
between our essential set and Fulton's essential set as originally
defined.

Given $w=w_1 w_2 \cdots w_n\in S_n$, draw an $n\times n$ matrix of ``bubbles'' $\circ$.  Replace
the bubbles in positions $(i, w_i)$ with an $\times$ and erase all bubbles in the ``hooks'' 
weakly below and (nonstandardly) to the \emph{left} of each $\times$. The 
{\bf diagram} of $w$, denoted ${\mathcal D}(w)$, consists of all bubbles left, which are those not in any hook. 
Under this convention $|{\mathcal D}(w)|=\ell(w_0 w)$. This reflects the fact that our diagram is
the left$\leftrightarrow$right mirror image of the \emph{standard} diagram of $w_0 w$;
see \cite[p.11]{FultonPragacz}. 
Fulton's essential set is then defined as the subset of ${\mathcal D}(w)$ such that
neither $(i+1,j)$ nor $(i,j-1)$ is in ${\mathcal D}(w)$. Let us denote Fulton's essential
set by $\Ess_{\rm Fulton}(w)$.  The desired bijection between bubbles in $\Ess_{\rm Fulton}(w)$
and $\Ess(w)$ sends the essential bubble with (row,column) indices $(r,s+1)$ to the
bigrassmannian $v_{r,s,t,n}$ where $t$ is the number of bubbles weakly above the essential
one in the same column.

For example, let $w=425163$. The figure below shows the positions $(i,w_i)$ marked with an $\times$,
and the bubbles in the diagram $D(w)$ shown as $\bullet$ or $\circ$ depending upon whether or not
they lie in $\Ess_{\rm Fulton}(w)$:
$$
\begin{matrix}
 &1&2&3&4&5&6 \\
1& & & &\times&\circ&\circ\\
2& &\times&\bullet& &\bullet&\circ\\
3& & & & &\times&\circ\\
4&\times& &\bullet& & &\bullet\\
5& & & & & &\times\\
6& & &\times& & & \\
\end{matrix}
$$
The following table then summarizes the bijection between the bubbles lying in Fulton's
essential set $\Ess_{\rm Fulton}(w)$, the essential Schubert conditions defining $X_w$,
and the bigrassmannians that comprise $\Ess(w)$.

\begin{tabular}{|c|c|c|}\hline
$(r,s+1) = $ (row,column)       & Schubert condition $C_{r,s,t}$:   & bigrassmannian $v=v_{r,s,t,n}$ \\
for bubble in $\Ess_{\rm Fulton}(w)$   & $\dim V_r \cap \CC^s \geq t$      & in $\Ess(w)$  \\ \hline\hline      
        &                              &        \\ 
$(2,3)$ & $\dim V_2 \cap \CC^2 \geq 1$ & 341256 \\ 
        &                              &        \\ \hline
        &                              &        \\ 
$(2,5)$ & $\dim V_2 \cap \CC^4 \geq 2$ & 152346 \\ 
        & (i.e. $V_2 \subset  \CC^4$)  &        \\ \hline
        &                              &        \\ 
$(4,3)$ & $\dim V_4 \cap \CC^2 \geq 2$ & 134526 \\ 
        & (i.e. $\CC^2 \subset  V_4$)   &        \\ \hline
        &                              &        \\ 
$(4,6)$ & $\dim V_4 \cap \CC^5 \geq 4$ & 123645 \\ 
        & (i.e. $V_4 \subset  \CC^5$)  &        \\ \hline
\end{tabular}

\end{example}

\subsection{Grassmannians and symmetric functions}
\label{LS-subsection}
  In looking for generators for the ideals $J_{v_{r,s,t,n}}$, we wish
to take advantage of symmetric function identities, so we briefly review here
the relation between symmetric functions and the Schubert calculus in type $A$.
We also point out how the calculations may be performed over $\ZZ$
rather than a coefficient field $\Bbbk$ of characteristic zero.

  Let $J:=S\setminus\{s_r\}$ where $s_r=(r \leftrightarrow r+1)$, so that
$W_J=S_r \times S_{n-r}$, and $G/P_J$ is the
{\it Grassmannian} of $r$-planes in $\CC^n$.  
The cohomology inclusion \eqref{parabolic-inclusion} or
\eqref{Hiller-parabolic-inclusion} 
remains valid working with coefficients in $\ZZ$
and becomes
$$
\begin{array}{lll}
H^\star(G/P_J)   &\cong H(G/B)^{W_J} &\hookrightarrow H(G/B)\\
\ZZ[\xx]^{W_J}/\ZZ[\xx]^W_+ \ZZ[\xx]^{W_J} 
             &\cong \left( \ZZ[\xx]/(\ZZ[\xx]^W_+) \right)^{W_J} 
                                 &\hookrightarrow \ZZ[\xx]/(\ZZ[\xx]^W_+).
\end{array}
$$
Here $\ZZ[\xx]:=\ZZ[x_1,\ldots,x_n]$ is viewed as the symmetric algebra $\ZZ[V]$, where $V$ is 
no longer the irreducible reflection representation of dimension $n-1$ for $W=S_n$ 
but rather the
natural permutation representation of dimension $n$.  In order to work over $\ZZ$,
one can replace the retraction in \eqref{parabolic-retraction}
with the {\bf Demazure operator} 
$$
\ZZ[\xx] \overset{\pi_{w_0(J)}}{\longrightarrow} \ZZ[\xx]^{W_J}
$$
associated to the longest element $w_0(J)$ in $W_J$, where
$$
\begin{aligned}
\pi_{s_i}(f) &:=\frac{x_i f-x_{i+1} s_i(f)}{x_i-x_{i+1}} 
\end{aligned}
$$
and $\pi_w :=\pi_{s_{i_1}} \cdots \pi_{s_{i_\ell}}$
if $w=s_{i_1} \cdots s_{i_\ell}$ is any reduced decomposition for $w$.

 In type $A_{n-1}$, one can replace the Hiller Schubert polynomial $S_w$ 
with {\bf Lascoux and} \linebreak {\bf Sch\"utzenberger's Schubert polynomial} $\schubp_w$ 
(see for example \cite{Macdonald, Manivel}):
one chooses the root linear functionals to be 
$x_i-x_j$ for $1 \leq i \neq j \leq n$ and
replaces the previous choice of $S_{w_0}= \prod_{i<j}(x_i-x_j)$ 
with an element which is equivalent modulo the ideal $(\ZZ[\xx]^W_+)$, namely
$$
\schubp_{w_0} :=  \xx^{\delta_n}:=x_1^{n-1} x_2^{n-2} 
\cdots x_{n-1}^1 x_n^0.
$$
Defining $\schubp_w := \partial_{w^{-1}w_0} \schubp_{w_0}$, property (c)
from Section~\ref{Hiller-section} tells us that the images of $S_w$
and $\schubp_w$ within $\ZZ[\xx]/(\ZZ[\xx]^W_+)$ are the same for all
$w\in W$. These $\schubp_w$:
\begin{enumerate}
\item[$\bullet$] lie in $\ZZ[\xx]$ and have nonnegative integer coefficients,
\item[$\bullet$] lift the cohomology classes $\sigma_w$ in the 
cohomology with {\it integer} coefficients 
\[H^\star(G/B,\ZZ) \cong \ZZ[\xx]/(\ZZ[\xx]^W_+),\] 
and
\item[$\bullet$] give us Schur functions in finite variable sets
whenever $w$ is grassmannian: if one has $\Des(w) \subseteq
\{(r,r+1)\}$, (in which case we say $u$ is $r$-{\bf grassmannian}), so that
$$
w_1 < w_2 < \cdots < w_r \text{ and }
w_{r+1} < w_{r+2} < \cdots < w_n,
$$
then
$$
\schubp_w = s_\lambda(x_1,\ldots,x_r),
$$
where $\lambda$ is the partition $\lambda=(w_r-r,\ldots,w_2-2,w_1-1)$.
Note that $\lambda$ has at most $r$ parts, all of size at most $n-r$,
so its Young diagram fits inside an $r \times (n-r)$ rectangle.
\end{enumerate}

In order both to suppress the variable set $x_1,\ldots,x_r$ from the
notation and to make more convenient use of symmetric function
identities, we will work within a quotient of the {\it ring of
  symmetric functions} with integral coefficients
$$
\Lambda=\Lambda_\ZZ(x_1,x_2,\ldots);
$$
see \cite[Chapter 1]{Macdonald-symm-fns}, \cite[Chapter 7]{Stanley-EC2}.
The $\ZZ$-basis for $\Lambda$ given by the Schur functions $s_\lambda$ 
has the property that the $\ZZ$-submodule $I_{r,n-r}$ spanned by all $s_\lambda$
with  $\lambda \not\subseteq (n-r)^r$ forms an ideal, and the
map sending $s_\lambda$ to $s_\lambda(x_1,\ldots,x_r)$ induces an isomorphism
$$
\Lambda/I_{r,n-r} \quad \overset{\sim}{\longrightarrow} \quad \ZZ[\xx]^{W_J}/\ZZ[\xx]^W_+ \ZZ[\xx]^{W_J} 
                     \quad \left( \cong H^\star(G/P_J,\ZZ) \right) 
$$
Thus $H^\star(G/P_J,\ZZ)$ has $\ZZ$-basis given by 
$$
\{ \sigma_w: w \in W^J \}
  = \{s_\lambda: \lambda \subseteq (n-r)^r \}.
$$

\subsection{A shorter presentation in type $A$}
\label{type-A-generators-subsection}

Starting with Theorem~\ref{general-J-generators}, our goal is to find
an even smaller set of generators for the ideal $J_{v_{r,s,t,n}}$
within the coinvariant algebra, so that through (\ref{eqn:I_wconcat})
we obtain an even shorter presentation of $I_w$ in type $A$.

First note that even though our proof of Theorem~\ref{general-J-generators}
for all finite Coxeter groups was done
in $\Bbbk[V]/(\Bbbk[V]^W_+)$ where the field $\Bbbk$ has characteristic zero,
the same proof works in type $A$ more generally for the {\it integral}
coinvariant algebra $H^\star(G/B,\ZZ)$. This follows since
the Schubert polynomials $\{\schubp_w\}$ satisfy the integer
coefficient versions of all of the requisite properties (a)-(g) used in Section~\ref{Hiller-section}.

The bigrassmannian $v_{r,s,t,n}$ described explicitly in \eqref{v-explicitly}
is $r$-grassmannian, and corresponds to the $j \times i$ 
rectangular partition $i^j$, where we define
$$
\begin{aligned}
i&:=s-t+1,\\
j&:=r-t+1.
\end{aligned}
$$

When $u$ and $v$ are $r$-grassmannian and correspond respectively to
partitions $\lambda$ and $\mu$, the Bruhat order relation $u \geq v$
is equivalent to inclusion $\lambda \supseteq \mu$ of their Young
diagrams, meaning that $\lambda_i \geq \mu_i$ for all $i$.  Thus
Theorem~\ref{general-J-generators} says that $J_{v_{r,s,t,n}}$ is
generated as an ideal of $\Lambda/I_{r,n-r}$ by
\begin{equation}
\label{step0-generators}
\{\schubp_u: \Des(u) = \{(r,r+1)\}, \,\, u \geq v_{r,s,t,n}\}
=\{ s_{\mu}: i^j \subseteq \mu \subseteq (n-r)^r \}.
\end{equation}

This presentation from Theorem~\ref{general-J-generators} can be
improved in type $A_{n-1}$ as follows:

\begin{theorem}
\label{J-generators}
  Given a bigrassmannian $v=v_{r,s,t,n}$ in type $A_{n-1}$ with $\schubp_v = s_{i^j}$, let
$$
\begin{aligned}
a&:=\min(n-r-i,r-j)\\
b&:=\min(i,j).
\end{aligned}
$$

Then $J_v$ is generated as an ideal of 
$H^{\star}(G/P_J, {\mathbb Z})\cong \Lambda/I_{r,n-r}$ by 
\begin{equation}
\label{generating-set-one}
\{ s_{\mu}: i^j \subseteq \mu \subseteq ((i+a)^b,i^{j-b}) \}.
\end{equation}

Alternatively, $J_v$ is generated by 
\begin{equation}
\label{generating-set-two}
\{ s_{\mu}: i^j \subseteq \mu \subseteq (i^j,b^a) \}.
\end{equation}
\end{theorem}

We delay our proof of Theorem~\ref{J-generators} until Section~\ref{minconjsection}.

\noindent
Note that in both of the asserted generating sets \eqref{generating-set-one} and \eqref{generating-set-two} for $J_v$,
the shapes $\mu$ indexing the generators $s_\mu$ run through an interval between
the $j \times i$ rectangular shape $i^j$ and the disjoint union of $i^j$ with a 
smaller rectangle of shape $a \times b$ or $b \times a$. 
In \eqref{generating-set-one} the smaller rectangle is to the right
of the rectangle $i^j$, with both top-justified, while in \eqref{generating-set-two} the 
smaller rectangle is below the rectangle $i^j$, with both left-justified.  
Thus in both cases, the generating sets have size $\binom{a+b}{a}$
and consist of generators whose multiset of degrees
have generating function 
$$
\sum_\mu q^{|\mu|} = q^{ij} \qbin{a+b}{a},
$$
where $\qbin{a+b}{a}$ is a $q$-{\it binomial coefficient} or {\it Gaussian polynomial} 
\cite[I.2 Exer. 3]{Macdonald}.

\begin{example}
\label{GR-comparison}
We examine the special case where the bigrassmannian 
$v:=v_{r,s,t,n}$ has $t$ equal to $r$ or $s$, so that the Schubert condition $C_{r,s,t}$
in \ref{Schubert-condition} becomes an {\it inclusion} 
$V_r \subseteq \CC^s$ or $V_r \supseteq \CC^s$.
Schubert varieties $X_w$ in type $A$ for which 
all Schubert conditions on $X_w$ take one of these two forms
were called Schubert varieties {\it defined by inclusions} in
\cite{GasharovReiner}.  That paper gave a presentation for the cohomology
containing
\begin{enumerate}
\item[$\bullet$]
for each inclusion condition $V_r \supseteq \CC^s$, 
a set of $s$ generators for $J_v$ of the form 
\begin{equation}
\label{first-kind-of-inclusion}
\{ e_m(x_1,\ldots,x_r) \}_{m=r-s+1,r-s+2,\ldots,r},
\end{equation}
\item[$\bullet$]
and for each inclusion condition $V_r \subseteq \CC^s$, 
a set of $r$ generators for $J_v$ of the form 
\begin{equation}
\{ e_m(x_r+1,\ldots,x_n) \}_{m=s-r+1,s-r+2,\ldots,n-r}.
\end{equation}
\end{enumerate}
We compare this with the presentation for $J_v$ 
in Theorem~\ref{J-generators},
say for the inclusion conditions of the form 
$V_r \supseteq \CC^s$, and using the 
generators given in \eqref{generating-set-two}.

Since $t=s$, one has 
$$
\begin{aligned}
i&=s-t+1=1\\
j&=r-s+1\\
b&=\min(i,j)=i=1\\
a&=\min(n-r-1,s-1)
\end{aligned}
$$
Hence  \eqref{generating-set-two} says
that $J_v$ is generated by the set of $a+1$ Schur functions 
$$
\{ s_\mu: 1^{r-s+1} \subseteq \mu \subseteq 1^{a+r-s+1} \}
=\{ e_m(x_1,\ldots,x_r) \}_{m=r-s+1,r-s+2,\ldots,a+r-s+1}
$$
which is exactly the first $a+1=\min(n-r,s)$ out of the $s$ generators 
listed in \eqref{first-kind-of-inclusion}.  Hence Theorem~\ref{J-generators}
provides a dramatic reduction in the size of the 
generating set for $J_v$ whenever $n-r$ is small compared to $s$.

We remark also that the techniques utilized in \cite{GasharovReiner}
seem very particular to the case where $X_w$ is defined by inclusions.
We do not know how to use them for some alternate approach
to the case of general $X_w$ considered in this paper.
\end{example}

\subsection{A minimality conjecture}
\label{minconj}
As we shall see in a moment, 
our generators for $I_w$ are not minimal. However,
we believe the following holds:

\begin{conjecture}
\label{minimality-conjecture}
The two generating sets for the ideal $J_{v_{r,s,t,n}}$ given in Theorem~\ref{J-generators}
are both minimal. 
\end{conjecture}

Via computer, we have verified this conjecture for all bigrassmannian permutations
where $r\leq 4$ and $n-r\leq 5$.

In fact, Conjecture~\ref{minimality-conjecture} indicates obstructions to short presentations of
$H^{\star}(X_w)$ in general. We now give 
a family of ideals that would require a large number of generators if the conjecture is true.

For a positive integer $m$, let $n=4m$, and consider in
$W=S_n=S_{4m}$ the bigrassmannian $v_{r,s,t,n}$ that corresponds
to $r=n-r=2m$ and $i=j=m$.  Then $a=b=m$, and $J_{v_{r,s,t,n}}(=I_{w_{r,s,t,n}})$ requires 
\begin{equation}
\label{eqn:badlowerbnd}
\binom{2m}{m} \sim \frac{4^m}{\sqrt{\pi m}} = \frac{\sqrt{2}^{n+2}}{\sqrt{\pi n}}
\end{equation}
generators according to Conjecture~\ref{minimality-conjecture}. 

The size of any minimal generating set of a homogeneous ideal is
well-defined. This is implied by the following well-known fact:

\begin{proposition} 
\label{prop:commalgfact}
Let $R$ be a commutative ring, and $\Lambda=\oplus_{n \geq 0}
\Lambda_n$ a graded, connected $R$-algebra, meaning that $\Lambda_0
=R$ and $\Lambda_i \Lambda_j \subset \Lambda_{i+j}$.  Let $M$ be a
graded $\Lambda$-module, with degrees bounded below, meaning that
$M=\oplus_{n \geq N} M_n$ for some $N \in \ZZ$, and $\Lambda_i M_j
\subset M_{i+j}$.

Then a set of homogeneous elements $\{m_i\}_{i=1}^t$ generate $M$ as a $\Lambda$-module
if and only if their images $\{\bar{m_i}\}_{i=1}^t$ span $M/\Lambda_+ M$ as a $R$-module.
In particular, $\{m_i\}_{i=1}^t$ form a minimal $\Lambda$-generating set with respect to
inclusion for $M$ if and only if $\{\bar{m_i}\}_{i=1}^t$ form a minimal $R$-spanning
set for $M/\Lambda_+ M$.  
\end{proposition}

In our setting, the well-definedness follows by setting
$\Lambda=\oplus_{n\geq 0}\Lambda_{n}$ to be the graded ring of
symmetric functions with ${\mathbb Z}$ coefficients and setting
$M=J_{v_{r,s,t,n}}$, so that the ${\mathbb Z}$-module $M/\Lambda_{+}M$
is a finitely generated abelian group.
%
%
Thus we conjecture that this abelian group $M/ \Lambda_+ M$ requires
$\binom{a+b}{a}$ generators, and in fact, we suspect that
$M/\Lambda_{+}M\cong {\mathbb Z}^{\binom{a+b}{a}}$.  So far a proof
has eluded us.

\begin{example}
\label{non-minimal-example}
Since $I_w=\sum_{v \in \Ess(w)} J_v$, and since we have
conjectured that the generating sets provided by Theorem~\ref{J-generators}
for $J_v$ are minimal, one might wonder whether their
concatentation gives a minimal generating set of $I_w$.  As mentioned above, this turns
out to be false in general.

The smallest counterexample is given by $w=1243$, which has
$$
\Ess(w) = \{v_1 = 2134, v_2=1324\}
$$
The generating sets given in Theorem~\ref{J-generators} for the ideals $J_{v_1}$ and $J_{v_2}$ are
$$
\begin{aligned}
J_{v_1} &= \langle s_{(1)}(x_1) \rangle 
       = \langle x_1 \rangle \\
       & \\
J_{v_2} &= \langle s_{(1)}(x_1,x_2), \quad s_{(2)}(x_1,x_2) \rangle 
           \quad \text{ or } \quad \langle s_{(1)}(x_1,x_2), \quad s_{(1,1)}(x_1,x_2) \rangle  \\
       &= \langle x_1+x_2, \quad x_1^2+x_1 x_2+x_2^2  \rangle
            \text{ or }\langle  x_1+x_2, \quad x_1 x_2 \rangle,
\end{aligned}
$$
and in each case they minimally generate their ideals $J_{v_i}$.
However, concatenating them gives non-minimal generating sets for $I_w$, namely
$$
\begin{aligned}
I_w = J_{v_1}+J_{v_2} &= \langle x_1, \quad x_1+x_2,\quad x_1^2+x_1 x_2+x_2^2 \rangle 
                      \quad \text{ or } \quad \langle x_1, \quad x_1+x_2, \quad x_1 x_2 \rangle  \\
                       &   \left( = \langle x_1, \quad  x_1+x_2 \rangle \right) .
\end{aligned}
$$
\end{example}

\begin{example}
Some readers may find the above earliest example artificial: 
though $X_w=X_{1243}$ lives inside $GL_4/B$, it is isomorphic to
$X_{21}=GL_2/B \cong \Proj^2$. However one can easily produce from
this more counterexamples with similar properties but no such
artificial nature.

For example, take $w=23541$, which has $\Ess(w) = \{v_1 = 31245, v_2=14235\}$.
Then $J_{v_1}$ and $J_{v_2}$ require one and two generators respectively, 
but the sum $I_w=J_{v_1}+J_{v_2}$ requires only two generators, not three.
\end{example}

\subsection{Some symmetric function identities}

The proof of Theorem~\ref{J-generators} on generators for $J_{v_{r,s,t,n}}$ will
use some symmetric function identities which we describe and prove in this
section.  We will make use of standard terminology, such as in
\cite{Macdonald-symm-fns, Sagan, Stanley-EC2}.

In particular, we will use the {\bf Pieri rule} expanding 
the product of an elementary symmetric function $e_r:=s_{1^r}$ with an
arbitrary Schur function into Schur functions:
\begin{equation}
\label{Pieri-rule}
e_k s_\lambda = \sum_\mu s_\mu,
\end{equation}
where the sum runs over all partitions $\mu$ obtained from $\lambda$ by adding on
a vertical strip of length $k$.  The following easy consequence
will be used in the proof of Theorem~\ref{J-generators} below.

\begin{lemma}
\label{Woo-lemma}
For any partition $\nu$ and nonnegative integer $k$, one has
\begin{equation}
\label{eqn:Woo-lemma}
s_{(\nu,1^k)} = \sum_{\ell=0}^k (-1)^{\ell} e_{k-\ell} \sum_\lambda s_\lambda,
\end{equation}
where the inner sum runs over partitions $\lambda$ 
obtained from $\nu$ by adding a horizontal strip of length~$\ell$.
\end{lemma}
\begin{proof}
Using the Pieri rule \eqref{Pieri-rule} to expand the right side 
of (\ref{eqn:Woo-lemma}), one obtains
$$
\sum_{\ell=0}^k (-1)^{\ell} e_{k-\ell} \sum_\lambda s_\lambda=
\sum_{(\ell,\lambda)}(-1)^\ell s_\lambda,
$$
where the sum runs over pairs $(\ell,\lambda)$ in which both $0 \leq
\ell \leq k$, and $\lambda$ is obtained from $\nu$ by first adding a
horizontal $\ell$-strip within the first $\ell(\nu)$ rows then adding
an arbitrary vertical $(k-\ell)$-strip.  Cancel all these pairs,
except for the one with $\ell=0$ and $\lambda=(\nu,1^k)$, via the
following sign-reversing involution: if $x$ (respectively $y$) is the
farthest east (respectively, farthest north) box in the horizontal
(respectively vertical) strip, then
\begin{enumerate}
\item[$\bullet$] when $y$ is to the right of $x$ (or when $\ell=0$ and
$\lambda \neq (\nu,1^k)$), move $y$ from the vertical to the
horizontal strip, and,
\item[$\bullet$] when $y$ is below $x$, move $x$ from the horizontal strip
to the vertical strip.
\end{enumerate}
\end{proof}

%
%
%
%

We also need the {\bf Jacobi-Trudi identity}: 
\begin{equation}
\label{Jacobi-Trudi}
s_\lambda = \det( h_{\lambda_i - i + j}  )_{i,j=1,2,\ldots,\ell(\lambda)}
\end{equation}
with the usual convention that $h_r := s_{(r)}$ for $r \geq 0$ and
$h_r = 0$ for $r < 0$.  This has the following consequence, also
to be used in the proof of Theorem~\ref{J-generators} below.

\begin{lemma}
\label{Woo-2nd-lemma}
Let $i < k$, and assume $\mu$ is a partition with
$
\mu_k > i \geq \mu_{k+1},
$
so that the $(i+1)^{st}$ column of the Young diagram for $\mu$ has length $k > i$.
Then
$$
s_\mu = \sum_{m=1}^k (-1)^{k-m} \,\, h_{\mu_m+k-i-m} \,\, s_{\mu^{(m)}},
$$
where for $m=1,2,\ldots,k$ one defines
$$
\mu^{(m)}:=(\mu_1,\mu_2,\ldots,\mu_{m-1},\widehat{\mu_m}, \mu_{m+1}-1,\mu_{m+2}-1,
               \ldots,\mu_k-1,i,\mu_{k+1},\mu_{k+2},\ldots,\mu_{\ell}),
$$
where $\ell:=\ell(\mu)$ and $\widehat{\mu_m}$ refers to the deletion of
the entry $\mu_m$.
%
\end{lemma}
\begin{proof}
Start with the $\ell \times \ell$
Jacobi-Trudi matrix for $\mu$.  From this create an $(\ell+1) \times \ell$ matrix by inserting a 
new row between its row $k$ and row $k+1$, having entries 
$$
(h_{i-k+1},h_{i-k+2},\ldots,h_{i-k+\ell}).
$$  
Then from this $(\ell+1) \times \ell$ matrix, create a singular $(\ell+1) \times (\ell+1)$ matrix 
by introducing an $(\ell+1)^{st}$ column that
duplicates the $(k-i)^{th}$ column.  This last duplicated column is
$$
\begin{array}{lccrll}
\ \ \ \ (h_{\mu_1+k-i-1},\ldots,h_{\mu_k-i},
   &h_{i-k + (k-i)},&h_{\mu_{k+1}-i-1},&\ldots, &h_{\mu_{\ell}+k-i-\ell}&)^T \\
= (h_{\mu_1+k-i-1},\ldots,h_{\mu_k-i},
   &1,            &0,              &\ldots,&0&)^T.
\end{array}
$$
Here we have used the facts that $h_0=1$ and that $h_{\mu_m+k-i-m}=0$ 
for $m \geq k+1$ because $\mu_m \leq \mu_{k+1} \leq i$ implies 
$\mu_m+k-i-m=(\mu_m-i)+(k-m)<0$.
One then checks that cofactor expanding the (zero) determinant of this 
$(\ell+1) \times (\ell+1)$ matrix along this duplicated column gives the asserted identity.
\end{proof}

\subsection{Proof of Theorem~\ref{J-generators}}
\label{minconjsection}
  The proof of the second statement will follow from the first, via the
well-known ring involution $\omega$ on symmmetric functions defined by
$$
\begin{aligned}
\Lambda &\overset{\omega}{\rightarrow} \Lambda \\
s_\lambda &\longmapsto s_{\lambda'}
\end{aligned}
$$
where $\lambda'$ is the conjugate partition to $\lambda$.  This means
that $\omega$ sends the ideal $I_{r,n-r}$ to the ideal $I_{n-r,r}$.
Hence the set \eqref{generating-set-two} generates $J_v$ within
$\Lambda/I_{r,n-r}$, where $v$ corresponds to an $i \times j$
rectangle, if and only if the set \eqref{generating-set-one} generates
the ideal $J_{v^\prime}$ within $\Lambda/I_{n-r,r}$, where $v^\prime$
corresponds to a $j\times i$ rectangle.

  The proof for \eqref{generating-set-one} is by induction on the
degree $d$, which is the number of boxes in our partition.
Our inductive hypothesis is that the portion of $J_v$ of degree at
most $d$ is generated by those elements of \eqref{generating-set-one}
of degree at most $d$, or equivalently, that all elements of
\eqref{step0-generators} of degree at most $d$ are writable in terms
of elements of \eqref{generating-set-one} of degree at most $d$.  The
base case, $d=ij$, is clear, since $s_{i^j}$ is the only element of
degree $ij$ in both sets.

Our proof for the inductive case proceeds in three steps.  Start with
the generating set for $J_v$ given in \eqref{step0-generators}.  We
wish to show that, modulo $I_{r,n-r}$, all such $s_\mu$ with $|\mu|=d$
lie in the ideal generated by those $s_\mu$ with $|\mu|<d$ and those
\begin{enumerate}
\item[{\sf Step 1.}]
with $\mu$ in the interval $[i^j,(n-r)^j]$, and then furthermore
\item[{\sf Step 2.}]
with $\mu$ in the interval $[i^j,(i+a)^j$], and then finally
\item[{\sf Step 3.}]
with $\mu$ in the interval $[i^j,((i+a)^b,i^{j-b})]$.
\end{enumerate}

\vskip.1in
\noindent
{\sf Step 1.}
  We will use induction on a certain partial order on partitions which depends on the index $j$. 
For a partition $\lambda$, define 
$$
\hat{\lambda}:=(\lambda_{j+1},\lambda_{j+2},\ldots),
$$
so that the Young diagram of $\hat{\lambda}$ consists of rows
$j+1,j+2,\ldots$ from the Young diagram of $\lambda$.
Then partially order the partitions containing $i^j$ by
decreeing $\lambda \prec_j \mu$ if either $|\hat{\lambda}| <|\hat{\mu}|$,
or if $|\hat{\lambda}| = |\hat{\mu}|$ but 
$\hat{\lambda}<\hat{\mu}$ in the {\it dominance order}, meaning that 
$$
\lambda_{j+1} + \lambda_{j+2} + \cdots \lambda_k \leq \mu_{j+1} + \mu_{j+2} + \cdots \mu_k
$$
for each $k \geq j$.

Now if $\mu$ does not already lie in the interval $[i^j,(n-r)^j]$, so
that $\ell(\mu) = k+j > j$, let $\nu$ be the partition obtained from
$\mu$ by removing $1$ from its last $k$ nonempty parts
$\mu_{j+1},\mu_{j+2},\ldots,\mu_{j+k}$.  Then by the Pieri rule
\eqref{Pieri-rule}, $e_ks_\nu=s_\mu+\sum_{\lambda} s_\lambda$, where
$\lambda$ runs through partitions other than $\mu$ obtained from $\nu$
by adding a vertical strip of size $k$.  One can check that any such
$\lambda$ satisfies $\lambda \prec_j \mu$: either the vertical strip
contains some boxes in the first $j$ rows, so that $|\hat{\lambda}|
<|\hat{\mu}|$, or if not, the location of the vertical strip forces
$\hat{\lambda}<\hat{\mu}$ in dominance.  Also, $k\geq1$, so
$|\nu|<|\mu|$.  Consequently, by induction on the order $\prec$, one
has an expression for $s_\mu$ showing that it is in the ideal
generated by $s_\lambda$ where either $|\lambda|<|\mu|$ or $\lambda$ is in
the interval $[i^j,(n-r)^j]$.

\vskip.1in
\noindent
{\sf Step 2.}  We will again use induction, this time on reverse dominance
  order.  We wish to write $s_\mu$ where $\mu$ is in the interval
  $[i^j,(n-r)^j]$ in terms of $s_\lambda$ where $|\lambda|<|\mu|$ or
  $\lambda$ lies in the interval $[i^j,(i+a)^j$].  Recall that
  $a=\min(n-r-i,r-j)$, and if $a=n-r-i$, then $n-r=i+a$ so there is
  nothing to do after Step 1.  Thus we may assume $a=r-j$.

 If $\mu$ does not already lie in the interval $[i^j,(i+a)^j]$, so
that $\mu_1 > i+a$, let $k:=\mu_1-i > a$, and let $\nu$ be the
partition obtained from $\mu$ by removing $1$ box from each of the
last $k$ nonempty columns in the Young diagram of $\mu$.  Note that
$\ell(\nu) \geq \ell(\mu)$ since $k<\mu_1$, and $\ell(\mu) \geq j$
since $i^j \subset \mu$.  Hence $(\nu,1^k)$ has length at least $j+k >
j+a = r$, so that the partition $(\nu,1^k) \not\subseteq (n-r)^r$, and
hence Lemma~\ref{Woo-lemma} tells us that
$$
\sum_{\ell=0}^k (-1)^\ell e_{k-\ell} \sum_\lambda s_\lambda \equiv 0 \mod I_{r,n-r}
$$
where in the sum $\lambda$ runs through partitions having no more than $j$ parts obtained
from $\nu$ by adding a horizontal strip of length $\ell$.  

  We claim that almost all of the terms in this sum, excepting the
single term with $\ell=k$ and $\lambda=\mu$, will have $|\lambda| <
|\mu|$ or $\lambda>\mu$.  If $\ell < k$, then $|\lambda|<|\mu|$.  If
$\ell = k$, note the horizontal strip $\lambda/\nu$ cannot have any
boxes in the first $i$ columns as those already have length $j$ in
$\nu$.  Therefore, the location of the horizontal strip forces
$\lambda > \mu$ in dominance.  Consequently, by induction, one has an
expression for $s_\mu$ showing that it is in the ideal generated by
$s_\lambda$ where $|\lambda|<|\mu|$ or $\lambda$ is in the interval
$[i^j,(i+a)^j]$.

\vskip.1in
\noindent
{\sf Step 3.}  Now we show that, if $\mu$ fits inside $(i+a)^j$ but
not inside $((i+a)^b,i^{j-b})$, then $s_\mu$ can always be written as
a sum of terms of the form $h_r s_\lambda$ for $r>0$ and $\lambda$
containing $i^j$.  Since $r>0$, $|\lambda|<|\mu|$, this suffices to
finish the proof.

Recall that $b=\min(i,j)$, and if $b=j$ then there is nothing to do
after Step 2. Thus we may assume $b=i < j$.

Let $k$ be the number of parts of $\mu$ which are strictly larger than
$i$, so that $k$ is the size of the $(i+1)^{st}$ column in the Young
diagram of $\mu$.  Since $\mu$ does not fit inside
$((i+a)^i,i^{j-i})$, it must be that $k > i$, and we are in the
situation of Lemma~\ref{Woo-2nd-lemma}.  Hence
$$
s_\mu = \sum_{m=1}^k (-1)^{k-m} h_{\mu_m+k-i-m}s_{\mu^{(m)}}
$$
where 
$$
\mu^{(m)}:=(\mu_1,\mu_2,\ldots,\mu_{m-1},\mu_{m+1}-1,\mu_{m+2}-1,
               \ldots,\mu_k-1,i,\mu_{k+1},\mu_{k+2},\ldots,\mu_{\ell})
$$
and $\ell=\ell(\mu)$.
Note that, since $\mu$ contains $i^j$, and hence $k \leq j$, each
$\mu^{(m)}$ also contains $i^j$.  Also note that each factor
$h_{\mu_m+k-i-m}$ has positive degree: $m \leq k$ implies $\mu_m \geq
\mu_k > i$, and hence $\mu_m+k-i-m = (\mu_m-i)+(k-m)> 0$. \qed

\section{A question}
\label{questions-section}


%
%
\begin{question}
Can one find a {\it minimal} generating set for the ideal $I_w$ in type $A_{n-1}$?

Can this at least be done for some of the 
recently-studied subclasses \cite{GasharovReiner, HultmanLinussonShareshianSjostrand, OhPostnikovYoo}
where $I_w$ can be generated by $n^2$ elements, such as
\begin{enumerate}
\item[$\bullet$]
when $X_w$  is {\it defined by inclusions}, which occurs when $w$ avoids the
patterns 
$$\{4231, 35142, 42513, 351624\},$$
\item[$\bullet$]
or more restrictively, when  $X_w$ is {\it smooth}, which occurs when $w$ avoids the
patterns 
$$\{3412,4231\}?$$
\end{enumerate}
\end{question}
\noindent
It was mentioned in the introduction that for
a special subclass of smooth Schubert varieties $X_w$ originally considered by Ding \cite{Ding1, Ding2},
there is a known minimal (in fact, complete intersection) presentation for $H^*(X_w,\ZZ)$ with $n$ 
relations that was exploited in \cite{DevelinMartinReiner}.  Short presentations would be
useful to extend that work further.

\section*{Acknowledgements}
The authors thank Nathan Reading for helpful comments and corrections.
We also thank the anonymous referee for pointing out the reference
\cite{ALP} to us, and for other helpful remarks.
VR is supported by NSF grant DMS-0601010. AY is supported by 
NSF grants DMS-0601010 and DMS-0901331. AW is
supported by NSF VIGRE grant DMS-0135345. This work was partially
completed while AY was a visitor at the Fields Institute in Toronto, and
was facilitated by a printer graciously provided by Lawrence Gray through the
University of Minnesota.

\end{document}